 \documentclass[11pt]{amsart}

\usepackage{a4}
\usepackage{amsfonts}
\usepackage{amsmath}
\usepackage{amsthm}
\usepackage{url}
\usepackage[all]{xy}
\usepackage{graphicx}
\usepackage{latexsym}
\usepackage{amssymb}
\usepackage[cp850]{inputenc}
\usepackage{epsfig}
\usepackage{psfrag}
\usepackage{amsfonts}
\usepackage{mathabx}
\usepackage{dsfont}

\newtheorem{sat}{Theorem}[section]		\newtheorem{lem}[sat]{Lemma}
\newtheorem{kor}[sat]{Corollary}			\newtheorem{prop}[sat]{Proposition}
				
\newtheorem*{defi*}{Definition}			\newtheorem*{bei*}{Example}
\newtheorem*{sat*}{Theorem}				\newtheorem*{kor*}{Corollary}
\newtheorem*{rmk*}{Remark}					

\let\ssection=\section
\renewcommand{\section}{\setcounter{equation}{0}\ssection}

\newtheorem*{namedtheorem}{\theoremname}
\newcommand{\theoremname}{testing}
\newenvironment{named}[1]{\renewcommand{\theoremname}{#1}\begin{namedtheorem}}{\end{namedtheorem}}

\theoremstyle{remark}
\newtheorem*{bem}{Remark}

\newcommand{\BR}{\mathbb R}			
\newcommand{\BN}{\mathbb N}			
\newcommand{\BS}{\mathbb S}			\newcommand{\BZ}{\mathbb Z}

\newcommand{\BONE}{\mathds 1}

		\newcommand{\CR}{\mathcal R}

\newcommand{\CW}{\mathcal W}

\newcommand{\actson}{\curvearrowright}

\newcommand{\D}{\partial}
\newcommand{\DD}{\nabla}

\DeclareMathOperator{\vol}{vol}		

\DeclareMathOperator{\diam}{diam}

\newcommand{\comment}[1]{}

\DeclareMathOperator{\Lip}{Lip}

\begin{document}

\title[]{Diameter and spectral gap for planar graphs}
\author{Larsen Louder and Juan Souto}
\thanks{Juan Souto has been partially supported by NSERC Discovery and Accelerator Supplement grants.}

\begin{abstract}
We prove that the spectral gap of a finite planar graph $X$ is bounded by $\lambda_1(X)\le C\left(\frac{\log(\diam X)}{\diam X}\right)^2$ where $C$ depends only on the degree of $X$. We then give a sequence of such graphs showing the the above estimate cannot be improved. This yields a negative answer to a question of Benjamini and Curien on the mixing time of the simple random walk on planar graphs.
\end{abstract}
\maketitle

\section{}
In this note we investigate the relationship between the diameters $\diam X$ of finite planar graphs $X$ and their spectral gaps, i.e.~the first non-zero eigenvalues $\lambda_1(X)$ of the associated combinatorial Laplacian. To avoid trivial cases, we will only consider connected graphs with $\diam X\ge 2$. Our first result is the following upper bound:

\begin{sat}\label{sat1}
For every $d$ there is $C_d$ with
$$\lambda_1(X)\le C_d\left(\frac{\log(\diam X)}{\diam X}\right)^2$$
for every finite connected planar graph $X$ with degree at most $d$.
\end{sat}

Theorem \ref{sat1} is not so surprising given that Spielman and Teng proved in \cite{ST} (see also \cite{KLPT}) that $\lambda_1(X)$ is bounded from above, up to a constant depending on the degree $d$ of $X$, by the reciprocal of the number of vertices:
\begin{equation}\label{eq1}
\lambda_1(X)\le C_d\frac 1{\vol X}
\end{equation}
In fact, Theorem \ref{sat1} follows easily from \eqref{eq1} and a simple computation. What may be more surprising, and this is the bulk of this paper, is that the bound provided by Theorem \ref{sat1} cannot be improved:

\begin{sat}\label{sat2}
There are positive numbers $d$ and $C$ and a sequence $(X_n)$ of pairwise distinct triangulations of $\BS^2$ of degree at most $d$ such that 
$$\lambda_1(X_n)\ge C\left(\frac{\log(\diam X_n)}{\diam X_n}\right)^2$$
for all $n$.
\end{sat}

\begin{bem}
Notice that \eqref{eq1} implies that for the graphs $X_n$ provided by Theorem \ref{sat2} we have $\vol X_n<(\diam X_n)^2$ for all $n$ large enough.
\end{bem}

Recall the following well-known relation between $\lambda_1(X)$ and the {\em mixing time} $\tau(X)$ of the simple random walk on $X$:
\begin{equation}\label{mixing}
\frac 1{C\lambda_1(X)}\le\tau(X)\le C\frac{\log(\vol X)}{\lambda_1(X)}
\end{equation}
Here $C$ is a constant which depends only on the maximal degree of $X$. We refer to the reader to \cite{LPW} for the definition of and facts on mixing times; see Theorem 12.3 and Theorem 12.4 in \cite{LPW} for a proof of \eqref{mixing}. 

In the light of \eqref{mixing} we can translate the upper bound on $\lambda_1(X)$ provided by  Theorem \ref{sat1} to a lower bound on the mixing time:

\begin{kor}
For every $d$ there is $C_d$ positive with
$$\tau(X)\ge C_d\left(\frac{\diam X}{\log(\diam X)}\right)^2$$
for every finite connected planar graph $X$ with degree at most $d$. \qed
\end{kor}

Similary we obtain from Theorem \ref{sat2}, the remark after the statement of the theorem, and the right inequality in \eqref{mixing} that:

\begin{kor}\label{noBC}
There are $d$ and $C$ positive and a sequence $(X_n)$ of pairwise distinct triangulations of $\BS^2$ with degree at most $d$ and such that 
$$\tau(X_n)\le C\frac{(\diam X_n)^2}{\log(\diam X_n)}$$
for all $n$.\qed
\end{kor}

Corollary \ref{noBC} yields a negative answer to a question asked by Benjamini and Curien \cite[Question 8]{BC}.
\smallskip

We now sketch the idea of the proof of Theorem \ref{sat2}. Given an expander $(Z_n)$ we consider, for each $n$, a graph $Y_n$ obtained by subdividing each edge of $Z_n$ into roughly $n^{10}$ edges. For each $n$ choose a point $p_n\in Y_n$ and let $A_n$ be a rotationally symmetric metric cylinder in which the $t$-th circle has length roughly equal to the number of points in $Y_n$ which are at distance $t$ from $p_n$. We prove that, under Neumann boundary conditions, $A_n$ satisfies the desired bound on the spectral gap and we deduce from the work of Mantuano \cite{Mantuano} that this implies that any discretization of $A_n$ does too. The desired graphs $X_n$ are constructed by completing triangulations of the annulus $A_n$ to triangulations of the sphere. 
\smallskip

Before concluding this introduction we would like to point out that Mantuano's work \cite{Mantuano} yields translations of Theorem \ref{sat1} and Theorem \ref{sat2} to the Riemannian setting. In the same spirit we wish to point out that \eqref{eq1} follows, again via \cite{Mantuano}, from the classical Yang-Yau theorem \cite{YY}.
\smallskip

\noindent{\bf Acknowledgements.} The second author is grateful for many conversations with Asaf Nachmias and Gourab Ray, without which this note would have never been possible. The second author thinks that the first author should thank him for all that beer. The first author thanks the second author for all that beer.
\smallskip

\noindent{\bf Notation:} Suppose that $(a_n)$ and $(b_n)$ are sequences of numbers. Throughout this paper we write $a_n\precapprox b_n$ (resp. $a_n\succapprox b_n$) if there is $C>0$ with $a_n<Cb_n$ (resp. $a_n>Cb_n$) for all $n$. If $a_n\precapprox b_n$ and $a_n\succapprox b_n$, then we write $a_n\approx b_n$.

\section{}

In this section we recall a few facts about Laplacians on graphs and Riemannian manifolds. We refer the reader \cite{Chavel,Chavel-little,Lubotzky} for details.

\subsection{}
Graphs will be denoted by capital letters $X,Y,Z,\dotsc$, often with super or subscripts. We denote by $V(X)$ the set of vertices and by $E(X)$ the set of edges of a graph $X$. We indicate by $x\sim y$ that two vertices of $X$ are joined by an edge. The valence, or degree, of a vertex $x\in V(X)$ is the number $\deg(x)$ of edges adjacent to $x$. The degree of a graph is the maximum of the degrees of its vertices. We assume, often implicitly, that all graphs are connected. 

The (combinatorial) Laplacian 
$$\Delta_X:\BR^{V(X)}\to \BR^{V(X)}$$
of a finite graph $X$ is the linear operator defined by
$$\Delta_X(f)(x)=\sum_{y\sim x}\left(f(x)-f(y)\right)$$
where we think of elements in $\BR^{V(X)}$ as functions on the finite set $V(X)$. With respect to the standard scalar product on $\BR^{V(X)}$, $\Delta_X$ is symmetric and positive semi definite and hence diagonalizable with spectrum
$$0=\lambda_0(X)<\lambda_1(X)\le\lambda_2(X)\le\dotsb$$
In this note we are only interested in the smallest positive eigenvalue $\lambda_1(X)$; this quantity is the {\em spectral gap} of the graph $X$. 

The {\em Rayleigh quotient} $\CR_X(f)$ of $f\in \BR^{V(X)}$ is defined to be
\begin{equation}\label{eq:ambrosia}
\CR_X(f)=\frac 12\frac{\sum_{x\sim y} (f(x)-f(y))^2}{\sum_{x\in V(X)}f(x)^2}
\end{equation}
Denoting by $\CW$ the space of all 2-dimensional linear subspaces of $\BR^{V(X)}$ we obtain the following characterization of $\lambda_1(X)$:
$$\lambda_1(X)=\min_{W\in\CW}\max_{f\in W}\CR_X(f)$$
This is the simplest incarnation of the so-called {\em minimax principle}. 

\begin{bem}
Let $f_1,f_2\in\BR^{V(X)}$ be non-zero elements and let $W\in\CW$ be their span. If $f_1,f_2$ have, when considered as functions, disjoint supports then 
$$\max_{f\in W}\CR_X(f)=\max\{\CR_X(f_1),\CR_X(f_2)\}$$
\end{bem}

\subsection{} So far we have considered graphs as purely combinatorial objects, but it will be useful to consider them also as 1-dimensional metric objects. To do so we identify each edge with the unit length interval $[0,1]$ and consider the induced inner distance on $X$. Associated to this distance we have the 1-dimensional Hausdorff measure $\lambda_X$ on $X$; this is just an arrogant way of referring to Lebesgue measure on edges. We denote by $\Lip(X)$ the set of Lipschitz functions $X\to\BR$ and identify $\BR^{V(X)}$ with the subset of $\Lip(X)$ consisting of functions which are linear on edges. 

Given an edge $e\simeq[0,1]$ and a point $x$ in the interior of $e$ we denote by $\DD f(x)$ the gradient of $f\in\Lip(X)$; it is defined almost everywhere by Rademacher's theorem. The Rayleigh quotient of $f\in\Lip(X)$ is defined to be:
\begin{equation}\label{eq:ambrosia2}
\CR_X(f)=\frac{\int_X\Vert\DD f(x)\Vert^2d\lambda_X(x)}{\int_X f(x)^2d\lambda_X(x)}
\end{equation}
Letting $\CW$ now denote the space of all 2-dimensional linear subspaces of $\Lip(X)$ we have again that
$$\lambda_1(X)=\min_{W\in\CW}\max_{f\in W}\CR_X(f)$$

\begin{bem}
  The factor of $\frac 12$ in \eqref{eq:ambrosia} is due to the fact
  that every edge is double counted. 
\end{bem}

 We remind the reader that a map
$$f:(X,d_X)\to(Y,d_Y)$$
between two metric spaces is an {\em $L$-quasi-isometry} if $f(X)$ is $L$-dense in $Y$ and if for all $x,x'\in X$ we have
$$\frac 1Ld_X(x,x')-L\le d_Y(f(x),f(x'))\le Ld_X(x,x')+L$$
Two metric spaces are {\em $L$-quasi-isometric} if there is an $L$-quasi-isometry between them. 
Quasi-isometric graphs have comparable spectral gaps \cite{Chavel-little}:

\begin{sat*}
For all $d$ and $L$ there is a constant $C$ such that the following holds: If $X$ and $Y$ are two $L$-quasi-isometric graphs of degree at most $d$, then
$$\frac 1C\lambda_1(X)\le\lambda_1(Y)\le C\lambda_1(X)$$
\end{sat*}

\subsection{} Suppose that $M$ is a compact Riemannian manifold with possibly non-empty totally geodesic boundary $\D M$, and recall that in the presence of boundary we say that a function on $f$ is smooth if it admits a smooth extension to some manifold with empty boundary containing $M$. The Laplacian of a function $f\in C^\infty(M)$ is defined to be
$$\Delta(f)=-\hbox{div}\DD(f)$$
where $\hbox{div}$ stands for the divergence and $\DD$ for the gradient. The Rayleigh quotient of $f\in C^\infty(M)$ is again defined by
$$\CR_M(f)=\frac{\int_M\Vert\DD f(x)\Vert^2d\vol_M(x)}{\int_M f(x)^2d\vol_M(x)}$$
where $d\vol_M$ is the volume form of $M$.

Denoting by $\nu$ the inward normal vector field along $\D M$ we have by Green's theorem that
\begin{equation}\label{green}
\int_Mf\Delta g d\vol-\int_M\langle\DD f,\DD g\rangle d\vol=\int_{\D M}f\langle\DD g,\nu\rangle d\vol
\end{equation}
for all $f,g\in C^{\infty}(M)$; here $\langle\cdot,\cdot\rangle$ stands for the Riemannian metric. 

Denote by $C^\infty_{N}(M)$ the space of smooth functions $g\in C^\infty(M)$ satisfying Neumann's boundary condition $\langle\DD g,\nu\rangle=0$. It follows from \eqref{green} that the restriction of the Laplacian to $C^\infty_{N}(M)$ is a self-adjoint operator. Moreover, its spectrum is discrete and non-negative
$$0=\lambda_0(M)<\lambda_1(M)\le\lambda_2(M)\le\dots$$
In particular, $L^2(M)$ admits a Hilbert basis $(\phi_i)_i\subset C^\infty_{N}(M)$ consisting of eigenfunctions of $\Delta$:
$$\Delta(\phi_i)=\lambda_i(M)\phi_i$$ 
Green's theorem \eqref{green} implies that if $f\in C_N^\infty(M)$ is a $\lambda_i(M)$ eigenfunction, then $\lambda_1(M)=\CR_M(f)$.

As it is the case for graphs, the smallest positive eigenvalue $\lambda_1(M)$ can be again computed via the minimax principle
\begin{equation}\label{minmax-N}
\lambda_1(M)=\min_{W\in\CW_{N}}\max_{f\in W}\CR_M(f)
\end{equation}
where this time $\CW_{N}$ is the set of all 2-dimensional linear subspaces of $C^\infty_{N}(M)$. We derive a second version of the minimax principle:

\begin{lem}\label{lem-minimax}
Suppose that $M$ is a compact Riemannian manifold with totally geodesic boundary and let $\lambda_1(M)$ be the first non-zero eigenvalue of the Laplacian on $M$ with Neumann boundary conditions. Then 
$$\lambda_1(M)=\min_{W\in\CW}\max_{f\in W}\CR_M(f)$$
where $\CW$ is the set of all 2-dimensional linear subspaces of $C^\infty(M)$.
\end{lem}
\begin{proof}
Let $M'$ be the double of $M$, $\sigma:M'\to M'$ the isometric involution with $M'/\sigma=M$, and  $\pi:M'\to M$ the associated orbit map. For $f\in C^\infty(M)$, the function $f'=f\circ\pi$ is Lipschitz and hence is the uniform limit of a sequence of smooth functions $f_n'\in C^\infty(M')$ such that 
\begin{equation}\label{eq-smooth}
\lim_{n\to\infty}\Vert\DD f'-\DD f_n'\Vert=0\ \ \hbox{in}\  L^2(M')
\end{equation}

Replacing $f_n'$ by $\frac 12(f_n'+f_n'\circ\sigma)$ we may assume that the functions $f_n'$ are $\sigma$ invariant and hence descend under $\pi$ to smooth functions $f_n\in C^{\infty}_{N}(M)$. From \eqref{eq-smooth} we obtain that
$$\lim_{n\to\infty}\CR_M(f_n)=\CR_M(f)$$
The claim follows now from \eqref{minmax-N}.
\end{proof}

\subsection{}
%
%

We say that a closed Riemannian manifold $M$ has $d$-bounded geometry if 
\begin{itemize}
\item $M$ has at most dimension $d$, 
\item the sectional curvature is pinched by $\vert\kappa_M\vert\le d$, and
\item the length of the shortest closed geodesic is bounded from below by $\frac 1d$.
\end{itemize}
If $M$ is a compact manifold with totally geodesic boundary then we say that $M$ has $d$-bounded geometry if its double does.

\begin{bem}
Recall that if $M$ is closed then an upper bound on the sectional curvature and a lower bound on the length of the shortest closed geodesic yield a lower bound on the injectivity radius.
\end{bem}

Discretizations of Riemannian manifolds have often been used in the literature (see e.g. \cite{Chavel-little}). The following is a well-known fact:

\begin{lem}\label{discretize}
For every $d$ there are $L$ and $d'$ such that every compact Riemannian manifold $M$ with $d$-bounded geometry admits a triangulation $T$ whose 1-skeleton $T^{(1)}$ is a graph of valence at most $d'$ which is $L$-quasi-isometric to $M$.\qed
\end{lem}

Abusing notation we will from now on make no distinction between triangulations and their $1$-skeleta.

To prove Theorem \ref{sat2} we will use the fact that quasi-isometric graphs and manifolds with bounded geometry have comparable spectral gaps. In \cite{Mantuano}, Mantuano proved such a result for closed manifolds via the minimax principle and a comparison of Rayleigh quotients of functions on the manifold and on the graph. The arguments used to compare Rayleigh quotients go through without any additional work in the presence of boundary. Moreover, since by Lemma \ref{lem-minimax} the first eigenvalue of the Laplacian with Neumann boundary conditions can be computed applying the minimax principle to {\em all} smooth functions on the manifold, we see that Mantuano's theorem applies also to manifolds with boundary:

\begin{sat*}[Mantuano \cite{Mantuano}]
For all $d$ and $L$ there is a constant $C$ such that the following holds: Suppose that $X$ is a graph with valence at most $d$, $M$ is a compact Riemannian manifold with $d$-bounded geometry and possibly non-empty boundary, and that $X$ and $M$ are $L$-quasi-isometric to each other, then
$$\frac 1C\lambda_1(X)\le\lambda_1(M)\le C\lambda_1(X)$$
where $\lambda_1(X)$ and $\lambda_1(M)$ are the first non-zero eigenvalues of the combinatorial Laplacian on $X$ and respectively the Laplacian with Neumann boundary conditions on $M$.
\end{sat*}

\section{}
In this section we prove Theorem \ref{sat1}; we also introduce some notation used later on. 

\subsection{}
Suppose that $X$ is a finite graph considered as a 1-dimensional metric space. For $p\in X$ we have the distance function
\begin{equation}\label{eq-distance}
\delta_p:X\to\BR,\ x\mapsto d(p,x)
\end{equation}
with range $[0,\diam_p(X)]$ where
\begin{equation}\label{eq-point-diam}
\diam_p(X)=\max\{d(p,x)\mid x\in X\}
\end{equation}
Notice that $\frac 12\diam(X)\le\diam_p(X)\le\diam(X)$.

The measure $(\delta_p)_*(\lambda_X)$ obtained by pushing the Lebesgue measure $\lambda_X$ of $X$ forward via $\delta_p$ is absolutely continuous with respect to the Lebesgue measure. We denote by 
$$\rho_{X,p}:\BR\to\BR^+$$
the Radon-Nicodym derivative of $(\delta_p)_*(\lambda_X)$ with respect to $dt$, meaning that for all continuous functions $f$ we have
\begin{equation}\label{eq:density}
\int_\BR f(t)\rho_{X,p}(t)dt=\int_X (f\circ\delta_p)(x)d\lambda_X(x)
\end{equation}
Notice that $\rho_{X,p}$ takes integer values and that $\rho_{X,p}(t)\ge 1$ for all $t\in[0,\diam_p(X)]$ and $\rho_{X,p}(t)=0$ otherwise. We state a useful observation:

\begin{lem}\label{tools}
Suppose that $X$ is a finite graph, $p\in X$ is a base point and 
$$F:[0,\diam_p(X)]\to\BR$$
is a Lipschitz function, then
$$\CR_X(F\circ\delta_p)=\frac{\int_0^{\diam_p(X)}F'(t)^2\rho_{X,p}(t)dt}{\int_0^{\diam_p(X)}F(t)^2\rho_{X,p}(t)dt}$$
\end{lem}

The proof of Lemma \ref{tools} is elementary and we leave it to the reader.\qed

\subsection{}
Armed with the notation we just introduced, we prove the following general fact about graphs:

\begin{prop}\label{construct-function}
For all $V$ and $r$ there is $C$ such that if $X$ is a finite graph satisfying $\vol X\le V(\diam X)^r$, then 
$$\lambda_1(X)\le C\left(\frac{\log(\diam X)}{\diam X}\right)^2$$
\end{prop}
\begin{proof}
Let $k\in\BZ$ be the integer part of $\log\frac{\diam(X)}2$ and notice that this implies that if $d(p_1,p_2)=\diam(X)$ then the balls $B(p_i,e^k)$ in $X$ of radius $e^k$ centered at $p_1$ and $p_2$ are disjoint. For $i=1,2$ we are going to construct, as long as the diameter of $X$ is over some threshold depending on $V$ and $r$, a function $f_i\in C^\infty(X)$ supported by $B(p_i,e^k)$ and with Rayleigh quotient 
$$\CR_X(f_i)\le (1+\log(r+2))\frac k{e^k}$$
Once this is done, the desired claim follows from the choice of $k$ because as we noticed earlier
$$\lambda_1(X)\le\max\{\CR_X(f_1),\CR_X(f_2)\}$$
For the sake of concreteness assume that $i=1$ and set $p=p_1$. Consider the distance function
$$\delta=\delta_{p}:X\to[0,\diam(X)]$$
and let $\rho=\rho_{X,p}$ be the function satisfying \eqref{eq:density}. Notice that
$$\int_0^{e^k}\rho(t)dt=\vol(B(p,e^k))\le eVe^{rk}$$
by the choice of $k$ and the assumption on the volume of $X$. Divide the interval $[0,e^k]$ into $k$ consecutive intervals $E_1,\dots,E_k$ of equal length $\frac{e^k}k$; set also $E_0=\emptyset$. Notice that 
$$\int_{E_1}\rho(t)dt=\vol(\delta^{-1}(E_1))\ge\frac{e^k}k$$
and suppose that for some $L$ we have 
$$\int_{E_{j+1}}\rho(t)dt\ge L\int_{E_{j+1}}\rho(t)dt$$
for all $j=1,\dots,k-1$. Then we get from the bound on $\vol X$ that
$$eVe^{rk}\ge\vol(\delta^{-1}(E_k))=\int_{E_{k}}\rho(t)dt\ge L^{k-1}\int_{E_{1}}\rho(t)dt\ge L^{k-1}\frac{e^k}k$$
It follows that (with finitely many graphs $X$ as possible exceptions) we have $L<\log(r+2)$. In particular, as long as $\diam(X)$ is over some threshold, there is $j\in\{1,\dots,k-1\}$ with
\begin{equation}\label{eq-bla1}
\int_{E_{j-1}}\rho(t)dt+\int_{E_{j+1}}\rho(t)dt\le (1+\log(r+2))\int_{E_{j}}\rho(t)dt
\end{equation}
Consider now the Lipschitz function $F:[0,\diam(X)]\to\BR$ given by
\begin{align*}
&F(t)=0 & \ \ \hbox{if}\ t\notin E_{j-1}\cup E_j\cup E_{j+1}\\
&F(t)=t-\frac{(j-1)e^{k}}k & \ \ \hbox{if}\ t\in E_{j-1}\\
&F(x)=\frac{e^k}k & \ \ \hbox{if}\ t\in E_{j}\\
&F(x)=\frac{(j+2)e^k}k-t & \ \ \hbox{if}\ t\in E_{j+1}
\end{align*}
and notice that 
\begin{align*}
&F'(t)=0 & \ \ \hbox{if}\ t\notin E_{j-1}\cup E_{j+1}\\
&\vert F'(t)\vert=1 & \ \ \hbox{if}\ t\in E_{j-1}\cup E_{j+1}
\end{align*}
The function $F\circ\delta$ is Lipschitz, supported by $B(p,e^k)$ and satisfies
\begin{equation}\label{eq-bla2}
\int_0^{\diam(X)} F'(t)^2\rho(t)dt=\int_{E_{j-1}}\rho(t)dt+\int_{E_{j+1}}\rho(t)dt
\end{equation}
On the other hand, we have
\begin{equation}\label{eq-bla3}
\int_0^{\diam(X)} F(t)^2\rho(t)dt\ge\int_{E_j}f(t)^2\rho(t)dt=\frac{e^k}k\int_{E_j}\rho(t)dt
\end{equation}
Combining Lemma \ref{tools} with equaltions \eqref{eq-bla1}, \eqref{eq-bla2} and \eqref{eq-bla3} we obtain for $f=F\circ\delta$ that
$$\CR_X(f)=\CR_X(F\circ\delta)\le(1+\log(r+2))\frac k{e^k}$$
Which is what we needed to prove.
\end{proof}

\subsection{}
Spielman and Teng \cite{ST} (see also \cite{KLPT}) proved:

\begin{sat*}[Spielman-Teng]
For every $d$ there is $C_d$ with
$$\lambda_1(X)\le C_1\frac 1{\vol X}$$
for every finite planar graph $X$ of degree at most $d$.
\end{sat*}

Theorem \ref{sat1} follows immediatelly from the Spielman-Teng theorem and Proposition \ref{construct-function}.:

\begin{named}{Theorem \ref{sat1}}
For every $d$ there is $C_d$ with
$$\lambda_1(X)\le C_d\left(\frac{\log(\diam X)}{\diam X}\right)^2$$
for every finite connected planar graph $X$ with degree at most $d$.\qed
\end{named}

\section{}

In section \ref{sec:meat} we will prove:

\begin{prop}\label{meat}
There is $d$ and a sequence $A_n$ of Riemannian surfaces homeomorphic to $\BS^1\times[0,1]$, with totally geodesic boundary and $d$-bounded geometry, and such that:
$$\vol(A_n)\approx n^{11},\ \diam(A_n)\approx n^{10}\log(n),\ \ \lambda_1(A_n)\approx n^{-20}$$
Moreover, each component of $\D A_n$ has unit length. 
\end{prop}

Recall that $\lambda_1(A_n)$ is the first positive eigenvalue of the Laplacian on $A_n$ with Neumann boundary conditions. 

Assuming Proposition \ref{meat}, we prove now Theorem \ref{sat2}:

\begin{named}{Theorem \ref{sat2}}
There are $d$ and $C$ positive and a sequence $(X_n)$ of pairwise distinct triangulations of $\BS^2$ with degree at most $d$ and such that 
$$\lambda_1(X_n)\ge C\left(\frac{\log(\diam X_n)}{\diam X_n}\right)^2$$
for all $n$.
\end{named}
\begin{proof}
Let $A_n$ be Riemannian surfaces provided by Proposition \ref{meat}, and for each $n$ let $T_n$ be a triangulation of $A_n$ as  provided by Lemma \ref{discretize}. Denoting by $D_1,D_2$ the cycles in $T_n$ corresponding to the boundary components of $A_n$, let $X_n$ be obtained from $T_n$ by coning off $D_1$ and $D_2$, each one of them to a different point. By construction $X_n$ is a triangulation of $\BS^2$. To see that $X_n$ has uniformly bounded degree observe that $T_n$ does by Lemma \ref{discretize} and that $D_1$ and $D_2$ have uniformly bounded many edges because both boundary components of $A_n$ have unit length. This also shows that $T_n$ and $X_n$ are uniformly quasi-isometric to each other. Since $T_n$ and $\Sigma_n$ are uniformly quasi-isometric to each other, we obtain 
\begin{equation}\label{eq-bennys1}
\diam X_n\approx \diam T_n\approx\diam A_n\approx n^{10}\log(n)
\end{equation}
\begin{equation}\label{eq-bennys2}
\lambda_1(X_n)\approx\lambda_1(T_n)\approx\lambda_1(\Sigma_n)\approx n^{-20}
\end{equation}
It follows from \eqref{eq-bennys1} and \eqref{eq-bennys2} that there is some $C$ with
$$\lambda_1(X_n)\ge C\left(\frac{\log(\diam X_n)}{\diam X_n}\right)^2$$
for all $n$, as we wanted to prove.
\end{proof}

\section{}

In this section we construct a family of graphs $Y_n$ needed in the proof of Proposition \ref{meat}. We stress that the graphs $Y_n$ are not planar.

\begin{lem}\label{models}
There is a sequence $(Y_n,p_n)$ of rooted graphs with vertices of degree at most $3$, with
$$\vol Y_n\approx n^{11},\ \ \diam Y_n\approx\log(n)n^{10},\ \ \lambda_1(Y_n)\approx{n^{-20}}$$
and such that the function $\rho_{Y_n,p_n}:\BR\to\BR_+$ satisfying \eqref{eq:density} has the following properties:
\begin{enumerate}
\item For all $t\in\BR$ we have $\rho_{Y_n,p_n}(t)\precapprox n$.
\item For any two $t,s\in\BR$ with $\vert t-s\vert<\frac 12$ we have $\vert\rho_{Y_n,p_n}(t)-\rho_{Y_n,p_n}(s)\vert\le 3$.
\end{enumerate}
\end{lem}

We explain briefly the meanings of the numbered statements in Lemma \ref{models}: (1) asserts that the cardinality of the distance sphere $\delta_{Y_n,p_n}=t$ in $Y_n$ centered at $p_n$ is bounded from above by some multiple of $n$, and (2) asserts that the number of points in the distance spheres $\delta_{Y_n,p_n}=t$ and $\delta_{Y_n,p_n}=t+h$ jumps up or down by at most $3$ if $h<\frac 12$.
\smallskip

We define some terms used in the proof of Lemma \ref{models}. Given a rooted graph $(Y,p)$ we say that $t\in[0,\diam_pY]$ is a {\em critical value} of the distance function $\delta_p:Y\to[0,\diam_pY]$ if the density $\rho_{Y,p}:\BR\to\BR_+$ is not continuous at $t$. Notice that $0$ and $\diam_pY$ are critical values and that all critical values belong to $\frac 12\BZ$, and recall that $\rho_{Y,p}$ takes values in $\BN$. A critical value $t$ is {\em good} if the step of discontinuity of $\rho_{Y,p}$ at $t$ is at most $3$. The critical value $0$ is good if and only if $\deg(p)\le 3$. Claim (2) of Lemma \ref{models} is that all critical values of the distance function associated to $(Y_n,p_n)$ are good.

Obviously the choice of the number $3$ in the definition of {\em good critical point} is somewhat arbitrary. It is tailored to the fact that we will work with graphs of valence at most $3$. Basically, one should think of good critical values in terms of general position; we hope that figure \ref{fig1} makes this remark clear.

\begin{figure}[h]
        \centering
         \includegraphics[width=6cm]{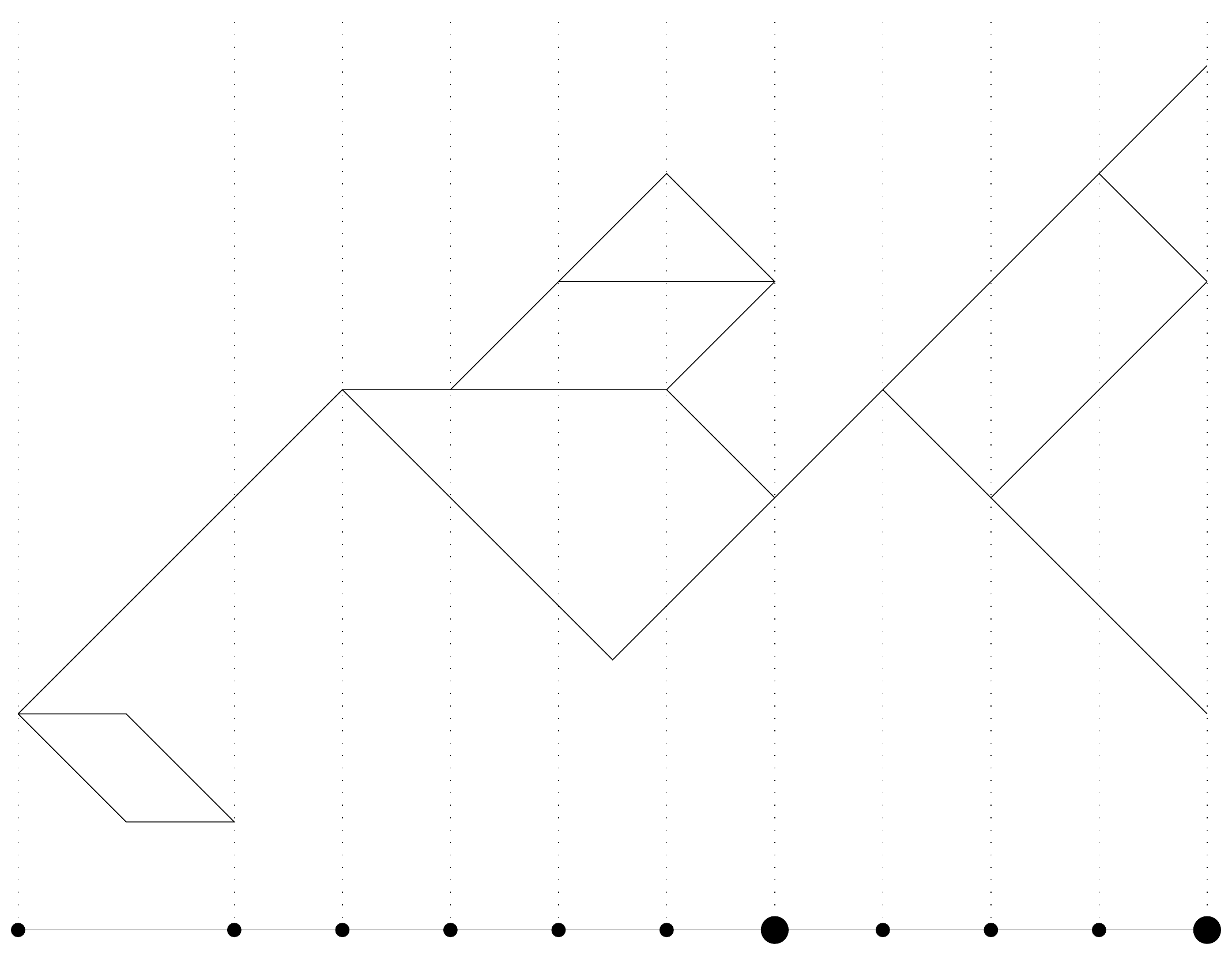}
         \caption{The distance function is the projection to the $x$-axes. The dots are the critical values; the good ones small and the bad ones large.}\label{fig1}
\end{figure}

\begin{bem}
Suppose that $Y$ is a finite graph of degree at most $3$, with no vertices of valence $1$, and that $p\in V(p)$ is a vertex. Both the number of critical points of $\delta_p:Y\to[0,\diam_pY]$ and $\max_{t\in[0,\diam_pY]}\rho_{Y,p}(t)$ are bounded by four times the number of trivalent vertices plus 2.
\end{bem}

\begin{proof}[Proof of Lemma \ref{models}]
Let $(Z_n)$ be an expander family consisting of trivalent graphs without vertices of valence $1$ and with $\vol Z_n\approx n$. That the sequence $(Z_n)$ is an expander means that there is $\epsilon>0$ with $\lambda_1(Z_n)\ge\epsilon$ for all $n$. It is known that this implies that $\diam(Z_n)\approx\log(n)$. In this section, every graph with an $n$ for subscript is going to be homeomorphic to $Z_n$. In particular, (1) in Lemma \ref{models} will be automatically satisfied by the remark above.

For each $n$ let $\hat Y_n$ be the graph obtained by subdividing each edge in $n^{10}$ edges. Metrically, this amounts to scaling the graph $Z_n$ by a factor $n^{10}$. It follows that
$$\vol \hat Y_n\approx n^{11},\ \ \lambda_1(\hat Y_n)\ge\frac{\epsilon}{n^{-20}},\ \ \diam\hat Y_n\approx\log(n)n^{10}$$
Choose now points $\hat p_n\in\hat Y_n$. It is easy to see that not all the critical values of the distance function $\delta_{\hat p_n}:\hat Y_n\to[0,\diam_{\hat p_n}\hat Y_n]$ are good, meaning that $(\hat Y_n,\hat p_n)$ does not satisfy (2) in the statement of Lemma \ref{models}. We are going to inductively perturb $\hat Y_n$ and construct 
$$(\hat Y_n^0,\hat p_n^0)=(\hat Y_n,\hat p_n),(\hat Y_n^1,\hat p_n^1),(\hat Y_n^2,\hat p_n^2),\dots$$
so that in every step more and more critical values are good. 

Suppose that we have constructed $(\hat Y_n^i,\hat p_n^i)$ with $\hat Y_n^i$ homeomorphic to $Z_n$, suppose that $t$ is the smallest bad critical value of $\delta_{\hat p_n^i}:\hat Y_n^i\to\BR$ and notice that $t>0$. Choose a point $x\in\hat Y_n^i$ at distance $\delta_{\hat p_n^i}(x)=t$ of the base point so that either $x$ is a trivalent vertex or a local maximum of $\delta_{\hat p_n^i}$. For each other point $y\in\hat Y_n^i$ with $\delta_{\hat p_n^i}(y)=t$ subdivide each edge adjacent to $y$ into two edges and let $\hat Y_n^{i+1}$ be the so obtained graph rooted at the vertex $\hat p_n^{i+1}$ corresponding to $\hat p_n^i$. Notice that $\rho_{\hat p_n^{i+1}}(s)=\rho_{\hat p_n^i}(s)$ for all $s<t$, that $t$ is still a critical value of $\rho_{\hat p_n^{i+1}}$, but that this time it is a good one. It follows that $\rho_{\hat p_n^{i+1}}$ has at least $i$ good critical values. Since $\hat Y_n^{i+1}$ is homeomorphic to $Z_n$ it follows that the total number of critical points is bounded by $4\vol Z_n+2$ and hence we have to repeat the process $N\le 4\vol Z_n+2$ times to end up with a graph 
$$(Y_n,p_n)=(\hat Y_n^N,\hat p_n^N)$$
for which the function $\delta_{p_n}:Y_n\to[0,\diam_{p_n}Y_n]$ has only good critical points; in particular (2) in Lemma \ref{models} holds. Moreover, since the graph $Y_n$ has been obtained from $Z_n$ by subdividing each edge somewhere between $n^{10}$ times and $n^{10}+4\vol Z_n+2$ times, we still have
$$\vol Y_n\approx n^{11},\ \ \lambda_1(Y_n)\approx n^{-20},\ \ \diam Y_n\approx\log(n)n^{10}$$
The validity of (1) follows from the fact that $Y_n$ is homeomorphic to $Z_n$ for all $n$.
\end{proof}

\section{}\label{sec:meat}

We can now prove Proposition \ref{meat}:

\begin{named}{Proposition \ref{meat}}
There is $d$ and a sequence $A_n$ of Riemannian surfaces homeomorphic to $\BS^1\times[0,1]$, with totally geodesic boundary and $d$-bounded geometry, and such that:
$$\vol(A_n)\approx n^{11},\ \diam(A_n)\approx n^{10}\log(n),\ \ \lambda_1(A_n)\approx n^{-20}$$
Moreover, each component of $\D A_n$ has unit length.
\end{named}

Let $(Y_n,p_n)$ be the sequence of rooted graphs provided by Lemma \ref{models}, set
$$R(n)=\diam_{p_n}(Y_n)$$
and consider the functions
$$\rho_{Y_n,p_n}:\BR\to\BR_+$$
satisfying \eqref{eq:density}. The support of $\rho_{Y_n,p_n}$ is $[0,R(n)]$ and recall that for every $t$ therein we have $\rho_{Y_n,p_n}(t)\ge 1$. On the other hand, we have $\rho_{Y_n,p_n}(t)\precapprox n$ by Lemma \ref{models} (1).

Lemma \ref{models} (2) implies that there is some positive constant $C$ such that for each $n$ there is a smooth function 
$$\sigma_n:\BR\to\BR_+$$
satisfying:
\begin{enumerate}
\item $C^{-1}\rho_{Y_n,p_n}(t)\le \sigma_n(t)\le C\rho_{Y_n,p_n}(t)$ for all $t\in[0,R(n)]$,
\item $\vert\sigma_n''(t)\vert\le C$ for all $t$, and 
\item $\sigma(t)=1$ for all $t\in[0,1]\cup[R(n)-1,R(n)]$.
\end{enumerate}
Denote by $\BS^1$ the circle of length $1$ and consider the cylinder 
$$A_n=([0,R(n)]\times\BS^1,dt^2+\sigma_n^2(t)d\theta^2)$$.
Each of the two boundary components of $\D A_n$ is totally geodesic of length $1$ and has a collar neighborhood isometric to a product. Observe that there is an isometric action $\BS^1\actson A_n$ whose orbits agree with the fibers of the projection
$$\pi:A_n\to[0,R(n)]$$
We claim that the surfaces $A_n$ have uniformly bounded geometry.

\begin{lem}\label{bounded-curvature}
There is $d$ such that $A_n$ has $d$-bounded geometry for all $d$. 
\end{lem}

\begin{proof}
We need to prove that the double $\Sigma_n$ of $A_n$ has $d$-bounded geometry for some $d$ and all $n$. To do this we need to estimate the sectional curvature of $\Sigma_n$ and the length of the shortest closed geodesic there in. Notice that by symmetry it suffices to bound the sectional curvature of $A_n$ at $(t,0)$. Since $A_n$ admits an isometric $\BS^1$-action whose orbit over $t\in[0,R(n)]$ has length $\sigma(t)$, it is classical that the sectional curvature at $(t,0)$ is given by 
$$\kappa_{A_n}(t,\theta)=\frac{\sigma_n''(t)}{\sigma_n(t)}$$ 
Since $\sigma_n(t)\ge C^{-1}$ and $\vert\sigma''_n(t)\vert\le C$ for all $t\in[0,R(n)]$ we deduce that $A_n$, and hence $\Sigma_n$, has sectional curvature bounded by $\vert\kappa_{\Sigma_n}\vert\le C^2$.

The action $\BS^1\actson A_n$ extends to an isometric action $\BS^1\actson\Sigma_n$ with associated Killing vector field $\frac{\D}{\D\theta}$. If $\gamma(t)=(T(t),\theta(t))$ is a geodesic in $\Sigma_n$ then $\langle\gamma'(t),\frac{\D}{\D\theta}\rangle$ is constant by Claireaux theorem. In particular, it follows $t\mapsto \theta(t)$ is monotone and hence that either $\gamma$ is orthogonal to the orbits of $\BS^1$ or has at least length $\min_{t}\sigma_n(t)\ge C^{-1}$. Since geodesics orthogonal to $\BS^1$-orbits have length at least $2R(n)$ we have proved that the surfaces $\Sigma_n$ have uniformly bounded geometry. 
\end{proof}

We estimate the diameter and volume of $A_n$.

\begin{lem}\label{lem-diam-vol}
$\vol A_n\approx n^{11}$ and $\diam A_n\approx \log(n)n^{10}$.
\end{lem}
\begin{proof}
The Riemannian volume form of $A_n$ is given by $d\vol_{A_n}=\sigma_n(t)dtd\theta$ in the coordinates $(t,\theta).$  In particular we have
$$\vol A_n=\int_0^{R(n)}\sigma_n(t)dt\approx \int_0^{R(n)}\rho_{Y_n,p_n}(t)dt=\vol Y_n\approx n^{11}$$
Here the second to last equality follows from the definition of $\rho_{Y_n,p_n}$ \eqref{eq:density}. The last statement is true by Lemma \ref{models}.

To estimate the diameter notice that the distance between points in $\pi^{-1}(t)$ is bounded by $\sigma_n(t)\le C\rho_{Y_n,p_n}(t)\precapprox n$ by assertion (1) in Lemma \ref{models}. On the other hand $A_n$ contains a geodesic arc of length exactly $R(n)=\diam_{p_n}(Y_n)\simeq\log(n)n^{10}$ intersecting every fiber of $\pi$ we deduce that
$$\diam A_n \precapprox \log(n)n^{10}$$
Finally, since the projection $\pi:A_n\to[0,R(n)]$ is 1-Lipschitz it follows that $\diam(A_n)\ge R(n)\succapprox \log(n)n^{10}$.
\end{proof}

Recall that $\lambda_1(A_n)$ is the first positive eigenvalue of the Laplacian on $A_n$ with Neumann boundary conditions. In order to prove Proposition \ref{meat} it remains to bound $\lambda_1(A_n)$; we start with the upper bound.

\begin{lem}\label{lem-up}
$\lambda_1(A_n)\precapprox n^{-20}$.
\end{lem}
\begin{proof}
By Lemma \ref{bounded-curvature}, the surfaces $A_n$ have $d$-bounded geometry for some uniform $d$. In particular, there are by Lemma \ref{discretize} constants $d'$ and $L$ such that for all $n$ the surface $A_n$ admits a triangulation $T_n$ of valence at most $d'$ and which is $L$-quasi-isometric to $A_n$. From Lemma \ref{lem-diam-vol} we obtain that $\diam T_n\approx \log(n)n^{10}$. Since $T_n$ is planar it follows from Theorem \ref{sat1} that
$$\lambda_1(T_n)\precapprox n^{-20}$$
Finally, using again that $A_n$ and $T_n$ are uniformly quasi-isometric to each other, it follows that $\lambda_1(A_n)\approx\lambda_1(T_n)$.
\end{proof}

It remains to bound $\lambda_1(A_n)$ from below. When doing so we will use the following fact:

\begin{lem}\label{trick}
Let $\phi:[0,R]\to(0,\infty)$ be a smooth function which is constant near the boundary and consider the surface
$$A=([0,R]\times\BS^1,dt^2+\phi^2d\theta^2)$$
If $\lambda_1(A)<\frac {4\pi^2}{\max\{\phi(t)^2\vert t\in[0,R]\}}$ then every $\lambda_1$-eigenfunction $f\in C^\infty_{N}(S)$ is constant on the orbits of the isometric action $\BS^1\actson A$, $(\eta,(t,\theta))\mapsto(t,\eta+\theta)$.
\end{lem}
\begin{proof}
If $f$ is a $\lambda_1$-eigenfunction, then for all $\eta\in\BS^1$ the function $f(t,\theta)=f(t,\theta+\eta)$ is also a $\lambda_1$-eigenfunction. In particular, the average 
$$\hat f(t,\theta)=\int_{\BS^1} f(t,\theta+\eta)d\eta$$
along the orbits $S(t)$ of the $\BS^1$-action is also a $\lambda_1$-eigenfunction. We need to prove that $f=\hat f$. Supposing that this is not the case, then $F=f-\hat f$ is non-zero, a $\lambda_1$-eigenfunction, and satisfies
$$0=\int_{S(t)}F(t,\theta)d\theta$$
for all $t$. Notice that this last equality implies that
$$\frac{4\pi^2}{\phi(t)^2}=\lambda_1(S(t))\le\frac{\int_{S(t)}\Vert\frac{\D F}{\D\theta}(t,\theta)\Vert^2d\vol_{S(t)}(\theta)}{\int_{S(t)}F(t,\theta)^2d\vol_{S(t)}(\theta) }$$
We estimate the Rayleigh quotient of $F$:
\begin{align*}
\int_A\Vert\DD F_{(t,\theta)}\Vert^2d\vol_A(t,\theta) 
& \ge \int_A\Vert\frac{\D F}{\D\theta}(t,\theta)\Vert^2d\vol_A(t,\theta)\\
& =\int_0^R\int_{S(t)}\Vert\frac{\D F}{\D\theta}(t,\theta)\Vert^2d\vol_{S(t)}(\theta) dt\\
& \ge \int_0^R\lambda_1(S(t))\int_{\BS^1_t}F(t,\theta)^2d\vol_{S(t)}(\theta) dt\\
& = \int_0^R\frac {4\pi^2}{\phi(t)^2}\int_{S(t)}F(t,\theta)^2d\vol_{S(t)}(\theta) dt\\
&\ge \frac {4\pi^2}{\max\{\phi(t)^2\vert t\in[0,R]\}}\int_0^R\int_{S(t)}F(t,\theta)^2d\vol_{S(t)}(\theta) dt\\
&\ge \frac {4\pi^2}{\max\{\phi(t)^2\vert t\in[0,R]\}}\int_AF(t,\theta)^2d\vol_A(t,\theta)
\end{align*}
Since $F$ is a $\lambda_1$-eigenfunction, we obtain from the computation above that
$$\lambda_1(A)=\CR_A(F)=\frac{\int_A\Vert\DD F_{(t,\theta)}\Vert^2d\vol_A(t,\theta)}{\int_AF(t,\theta)^2d\vol_A(t,\theta)}\ge \frac {4\pi^2}{\max\{\phi(t)^2\vert t\in[0,R]\}}$$
contradicting our assumption.
\end{proof}

We are now ready to prove Proposition \ref{meat}:

\begin{proof}[Proof of Proposition \ref{meat}]
Having proved Lemma \ref{bounded-curvature}, Lemma \ref{lem-diam-vol} and Lemma \ref{lem-up}, it remains to bound $\lambda_1(A_n)$ from below. 

By construction, the fibers of the projection $\pi:A_n\to[0,R(n)]$ are the orbits of the isometric circle action $\BS^1\actson A_n$. Moreover, as we already observed when proving Lemma \ref{lem-diam-vol}, each such fiber has length $\precapprox n$. In particular, we obtain from Lemma \ref{trick} that every $\lambda_1(A_n)$-eigenfunction $f\in C^\infty_{N}(A_n)$ is constant along the fibers of $\pi$. This means that there is a function
$$F:[0,R(n)]\to\BR$$
with $f=F\circ\pi$. Consider the 2-dimensional space $W\subset C^\infty([0,R(n)])$ spanned by the restriction of $F$ to $[0,R(n)]$ and the constant function $\BONE(x)=1$. We can summarize much of what we have said so far in the following equation:
\begin{equation}\label{eq123}
\lambda_1(A_n)=\max\{\CR_{A_n}(\phi\circ\pi)\vert \phi\in W\}
\end{equation}
Recall that on the other hand we have the distance function
$$\delta_{p_n}:Y_n\to[0,\diam_{p_n}(Y_n)]$$
and by the minimax principle we have
\begin{equation}\label{eq1234}
\lambda_1(Y_n)\le \max\{\CR_{Y_n}(\phi\circ\delta_p)\vert \phi\in W\}
\end{equation}
We estimate $\CR_{A_n}(\phi\circ\pi)$ for $\phi\in W$:
\begin{align*}
\int_{A_n}\Vert\DD (\phi\circ\pi)\Vert^2d\vol_S(t,\theta)
&=\int_{A_n} \phi'(t)^2d\vol_S(t,\theta)\\
&=\int_0^{R(n)} \phi'(t)^2\sigma_n(t) dt\\
&\ge\frac 1C\int_0^{R(n)} \phi'(t)^2\rho_{Y_n,p_n}(t) dt
\end{align*}
\begin{align*}
\int_{A_n}(\phi\circ\pi)^2d\vol_S(t,\theta)
&=\int_{A_n} \phi(t)^2d\vol_S(t,\theta)\\
&=\int_0^{R(n)} \phi(t)^2\sigma_n(t) dt\\
&\le C\int_0^{R(n)} \phi(t)^2\rho_{Y_n,p_n}(t) dt
\end{align*}
Combining these two inequalities with Lemma \ref{tools} we obtain
$$\CR_{A_n}(\phi\circ\pi)\ge\frac 1{C^2}\CR_{Y_n}(\phi\circ\delta_{p_n})$$
for all $\phi\in W$. In particular \eqref{eq123}, \eqref{eq1234} and Lemma \ref{models} imply that 
$$\lambda_1(A_n)\ge\frac 1{C^2}\lambda_1(Y_n)\succapprox n^{-20}$$
Which concludes the proof of Proposition \ref{meat}.
\end{proof}

\bigskip

\noindent Department of Mathematics, University of Michigan.
\newline \noindent
\texttt{llouder@umich.edu}

\bigskip

\noindent Department of Mathematics, University of British Columbia.
\newline \noindent
\texttt{jsouto@math.ubc.ca}


\begin{thebibliography}{10}

\bibitem{BC}
I. Benjamini and N. Curien, {\em On limits of graphs sphere packed in Euclidean space and applications}, European Journal of Combinatorics 32 (2011).


\bibitem{Chavel}
I. Chavel, {\em Eigenvalues in Riemannian Geometry}, Academic Press, New York, 1984.

\bibitem{Chavel-little}
I. Chavel, {\em Isoperimetric Inequalities}, Cambridge University Press, 2001. 

\bibitem{KLPT}
J. Kelner, J. Lee, G. Price and S. Teng, {\em Metric uniformization and spectral bounds for graphs}, 
Geom. Funct. Anal. Vol. 21 (2011).

\bibitem{LPW}
D. Levin, Y. Peres and E. Wilmer, {\em Markov Chains and Mixing Times}, AMS (2008).

\bibitem{Lubotzky}
A. Lubotzky, {\em Discrete groups, expanding graphs and invariant measures}, Progress in Mathematics, vol. 125, Birkh\"auser, 1994.

\bibitem{Mantuano}
T. Mantuano, {\em Discretization of Compact Riemannian Manifolds Applied to the Spectrum of Laplacian}, Annals of Global Anal. and Geom. 27, 2005.

\bibitem{ST}
D. Spielman and S. Teng. {\em Spectral partitioning works: planar graphs and finite element meshes}, 37th Annual Symposium on Foundations of Computer Science, IEEE Comput. Soc. Press, 1996.

\bibitem{YY}
P. Yang and S. Yau, {\em Eigenvalues of the laplacian of compact Riemann surfaces and minimal submanifolds}, Ann. Scoula Norm. Sup. Pisa 7 (1980).

\end{thebibliography}
\end{document}